\documentclass[11pt,a4paper]{article}
\usepackage{etex}
\usepackage{inputenc}
\usepackage[english]{babel}
\usepackage{lmodern}					
\usepackage{textcomp}				
\usepackage{verbatim}			
\usepackage[toc,page]{appendix} 		
\usepackage{cite}
\usepackage{amsmath,					
			amsthm,
			amsfonts,
			amssymb, 
			mathtools}
\usepackage{stmaryrd}
\usepackage{physics}
\usepackage{enumitem}
\usepackage{hyperref}
\hypersetup{
colorlinks=true,
linkcolor=purple,
filecolor=magenta,
urlcolor=cyan,
citecolor=blue
}

 \usepackage{tikz-cd} 
\usetikzlibrary{tikzmark}
\usetikzlibrary{arrows}
 \usepackage{verbatim}
\definecolor{DarkGreen}{HTML}{1cad22}
\usepackage[hmargin=3cm,vmargin=2.5cm]{geometry}
\usepackage{graphicx} 
\usepackage{amsthm, amssymb, amsmath}
\usepackage{color}
\usepackage{xcolor}
\usepackage{tikz}
\usetikzlibrary{positioning}
\usepackage{tikz-cd}
\usepackage{amsmath}
\usepackage{amssymb}
\usepackage{float}
\usepackage{caption}
\usepackage{mathtools}
\usepackage{hyperref}
\usepackage{mathrsfs}

\usepackage{graphicx}
\usepackage{epsfig}
\newcommand{\HH}{\ensuremath{\mathbb H}}

\numberwithin{equation}{section}

{\theoremstyle{definition}\newtheorem{definition}{Definition}[section]

\newtheorem{defititle}[definition]{\Title}

\newtheorem{remark}[definition]{Remark}
\newtheorem{example}[definition]{Example}

}
\newtheorem{prop}[definition]{Proposition}
\newtheorem{proposition-definition}[definition]{Proposition-Definition}
\newtheorem{lemma}[definition]{Lemma}
\newtheorem{thm}[definition]{Theorem}
\newtheorem{cor}[definition]{Corollary}

\newtheorem*{theorem}{Theorem}

\newtheoremstyle{named}{}{}{\itshape}{}{\bfseries}{.}{.5em}{\thmnote{#3's }#1}
\theoremstyle{named}

\usepackage{amsthm}

\newenvironment{mainthm-env}[1]
  {\begin{theorem}[Main Theorem #1]}
  {\end{theorem}}

\title{Equivariant rigidity of complex and quaternionic moment--angle manifolds}

\author{Ioannis Gkeneralis\footnote{Department of Mathematics, Aristotle University of Thessaloniki, \href{mailto:igkeneralis@math.auth.gr}{igkeneralis@math.auth.gr}}}

\date{ }

\newcommand{\keywords}[1]{\textbf{Keywords: } #1}
\newcommand{\msc}[1]{\textbf{Mathematics Subject Classification: } #1}

\begin{document}

\vspace{-0.5cm}

\maketitle

\begin{abstract}
We investigate equivariant rigidity properties of complex and quaternionic moment--angle manifolds. By reducing the classification problem to the equivariant rigidity of their quasitoric or quoric quotients and by using the associated principal bundle structures, we establish rigidity results within the category of locally linear actions. We prove that complex moment--angle manifolds are equivariantly rigid up to homeomorphism: any closed locally linear manifold equivariantly homotopy equivalent to a complex moment--angle manifold is equivariantly homeomorphic to it. In the quaternionic setting, under Hopkinson's globality condition on the characteristic data, we obtain the analogous equivariant homeomorphism rigidity for quaternionic moment--angle manifolds. These results show that, in the equivariant category considered here, the equivariant homotopy type of a moment--angle manifold determines its equivariant homeomorphism type.
\end{abstract}

\noindent \keywords{moment--angle manifold, quasitoric manifold, quoric manifold, equivariant rigidity, polyhedral product, principal torus bundle} \\[1ex]
\msc[2020]{Primary  57S25; Secondary 57R19, 57R18, 55U10 }

\vspace{0.5ex}
\noindent \textbf{Conflict of Interest and Data Availability Statement:} The author states that there is no conflict of interest to declare. Moreover, data sharing is not applicable to this article as no datasets were generated or analyzed during the research carried out in this paper.

\tableofcontents

\section{Introduction}

Rigidity phenomena play a central role in geometric topology. The general question is to what extent the homotopy type of a manifold determines its homeomorphism or diffeomorphism type. Classical examples include the Borel Conjecture for aspherical manifolds and rigidity results related to the Farrell--Jones Conjecture \cite{luckreich}. When manifolds carry additional geometric or combinatorial structures, it is natural to ask analogous rigidity questions in equivariant or stratified settings.

\medskip

Toric topology provides a particularly rich framework for studying such questions. In their foundational work, Davis and Januszkiewicz \cite{dj} introduced quasitoric manifolds as topological counterparts of toric varieties. A quasitoric manifold is a smooth $2n$--manifold equipped with a locally standard action (coordinatewise multiplication) of the torus $T^n$ whose orbit space is a simple polytope. The topology of these manifolds is closely tied to the combinatorics of the polytope together with an associated characteristic function. Rigidity questions for quasitoric manifolds have been studied extensively (see \cite{dav2}  or \cite{Wiemeler} and references therin); in particular, Metaftsis and Prassidis \cite{mp} proved that quasitoric manifolds satisfy equivariant topological rigidity within the class of locally linear torus manifolds.

\medskip

Closely related to quasitoric manifolds are the moment--angle manifolds introduced by Davis and Januszkiewicz \cite{dj} and further developed by Buchstaber and Panov \cite{bp}. Given a simple polytope $P$ with $m$ facets, the associated moment--angle manifold $\mathcal Z_P$ carries a natural action of the torus $T^m$. These manifolds admit several equivalent descriptions: they can be realized as intersections of real quadrics, and more conceptually as polyhedral products $(D^2,S^1)^K$ determined by the combinatorics of the nerve complex $K$ of $P$ (see \cite{bbcg,bp}). A fundamental feature of moment--angle manifolds is the existence of a free action of a subtorus
\[
K \cong (S^1)^{m-n},
\]
whose quotient is a quasitoric manifold over $P$. Thus moment--angle manifolds arise as total spaces of principal torus bundles over quasitoric manifolds and play a central role in the interaction between torus actions \cite{bp}, polyhedral geometry \cite{bbcg}, and homotopy theory \cite{gbricTheriault}.

\medskip

More recently, quaternionic analogues of these constructions have been studied. Replacing the complex torus $(S^1)^n$ with the quaternionic torus $(S^3)^n$, Hopkinson \cite{ho} introduced quaternionic moment--angle manifolds together with their quotients, called \emph{quoric manifolds}, which serve as quaternionic analogues of quasitoric manifolds. These manifolds admit locally regular actions (conjugation) of $(S^3)^n$ whose orbit spaces are simple polytopes. Their geometry reflects the representation theory of $S^3$ and parallels many features of the classical complex toric setting. Rigidity questions in this quaternionic context have also been investigated: Prassidis and the author \cite{gp} proved that quoric manifolds satisfy equivariant topological rigidity.

\medskip

The purpose of this paper is to study rigidity properties of moment--angle manifolds from this perspective. More precisely, we ask whether a manifold that is equivariantly homotopy equivalent to a moment--angle manifold must be equivariantly homeomorphic to it. We have two situations in mind; the classical complex moment--angle manifolds and their quaternionic analogues.

\medskip

Our main results establish equivariant homeomorphism-type rigidity for moment--angle manifolds in both the complex and quaternionic settings. We formulate the result in a unified way. Let \(G=T^m\) in the complex case and \(G=Q^m=(S^3)^m\) in the quaternionic case. Let \(\mathcal Z\) denote \(\mathcal Z_P\) in the complex case and \(\mathcal Z_P^{\mathbb H}\) in the quaternionic case.

The result depends on the choice of characteristic data giving a freely acting kernel subgroup \(K\subset G\) such that the quotient space \(\mathcal Z/K\) is a quasitoric manifold in the complex case and a quoric manifold in the quaternionic case. In the quaternionic case we assume, following Hopkinson \cite{ho}, that the characteristic data are \emph{global}. In particular, we prove the following result.

\begin{theorem}[Theorem~\ref{thm:main_rigidity}]
Let \(\mathcal Z\) be a complex or quaternionic moment--angle manifold associated to a simple polytope \(P^n\), equipped with its canonical \(G\)-action. Suppose that \(P\) admits characteristic data as above, with global characteristic data in the quaternionic case. If \(N\) is a closed locally linear \(G\)-manifold and \(f:N\longrightarrow \mathcal Z \) is a \(G\)-equivariant homotopy equivalence, then \(N\) is \(G\)-equivariantly homeomorphic to \(\mathcal Z\).
\end{theorem}

Thus the equivariant homotopy type of \(\mathcal Z\), within the class of closed locally linear \(G\)-manifolds, determines its equivariant homeomorphism type. We emphasize that this is a homeomorphism-type rigidity statement. We do not claim that the given map \(f\) is \(G\)-equivariantly homotopic to a homeomorphism.

As a consequence, the complex and quaternionic cases give the following results.

\begin{theorem}[Complex case~\ref{cor:rigidity_complex}]
Let \(P^n\) be a simple polytope equipped with characteristic data giving a free kernel action on \(\mathcal Z_P\) whose quotient is a quasitoric manifold. If \(N\) is a closed locally linear \(T^m\)-manifold and \(f:N\longrightarrow \mathcal Z_P\) is a \(T^m\)-equivariant homotopy equivalence, then \(N\) is \(T^m\)-equivariantly homeomorphic to \(\mathcal Z_P\).
\end{theorem}

\begin{theorem}[Quaternionic case~\ref{cor:rigidity_quaternionic}]
Let \(P^n\) be a simple polytope admitting a global characteristic pair. If \(N\) is a closed locally linear \(Q^m\)-manifold and \(f:N\longrightarrow \mathcal Z_P^{\mathbb H}\) is a \(Q^m\)-equivariant homotopy equivalence, then \(N\) is \(Q^m\)-equivariantly homeomorphic to \(\mathcal Z_P^{\mathbb H}\).
\end{theorem}

The proof is based on a reduction to the rigidity of the quotient manifolds. The kernel subgroup \(K\) acts freely on the model moment--angle manifold \(\mathcal Z\), and we show that it also acts freely on any closed locally linear \(G\)-manifold \(N\) that is \(G\)-equivariantly homotopy equivalent to \(\mathcal Z\). The map \(f:N\to\mathcal Z\) therefore descends to a \(G/K\)-equivariant homotopy equivalence
\[
    \bar f:N/K\longrightarrow \mathcal Z/K.
\]
The quotient \(N/K\) is shown to be a closed locally linear manifold of the correct dimension: \(2n\) in the complex case and \(4n\) in the quaternionic case. Hence the equivariant rigidity theorems for quasitoric manifolds \cite{mp} and quoric manifolds \cite{gp} apply to the quotient map \(\bar f\).

The final step is bundle-theoretic. Both \(N\) and \(\mathcal Z\) are principal \(K\)-bundles over their quotients. The map \(f\) identifies \(N\) with the pullback bundle \(\bar f^{\,*}\mathcal Z\). The rigidity theorem for the quotient provides a \(G/K\)-equivariant homeomorphism
\[
    h:N/K\longrightarrow \mathcal Z/K
\]
which is equivariantly homotopic to \(\bar f\). Equivariant homotopy invariance of pullbacks of numerable principal bundles then gives
\[
    \bar f^{\,*}\mathcal Z\cong_G h^*\mathcal Z.
\]
Since \(h\) is a homeomorphism, \(h^*\mathcal Z\) is
\(G\)-equivariantly isomorphic to \(\mathcal Z\).  Thus
\[
    N
    \cong_G
    \bar f^{\,*}\mathcal Z
    \cong_G
    h^*\mathcal Z
    \cong_G
    \mathcal Z,
\]
which proves the main theorem.

\medskip

\paragraph{Organization of the paper.}
The paper is organized as follows.  In Section~\ref{sec:moment_angle} we recall the construction of complex and quaternionic moment--angle manifolds and their quasitoric and quoric quotients.  Section~\ref{sec:reduction} explains the reduction from rigidity of moment--angle manifolds to the known equivariant rigidity results for quasitoric and quoric manifolds. In Section~\ref{sec:classification} we introduce the pullback construction for the associated principal kernel bundles and prove the main equivariant homeomorphism-type rigidity theorem.

\medskip

\paragraph{Acknowledgments.}
The author would like to thank S.~Prassidis for suggesting the problem of topological rigidity of moment--angle manifolds and for many helpful discussions during the preparation of this paper. The author is also grateful to J.~Daura Serrano and J.~Garriga Puig for numerous discussions on quaternionic moment--angle manifolds. The material on quaternionic moment--angle manifolds presented here forms part of a joint work that will appear elsewhere. Many thanks also must go to C.~Aravanis, P.~Batakidis, R.~Tsiavou, and B.~Schutte for their interest in this project.

The author also thanks A.~Bahri for pointing out the relation with generalized moment--angle complexes associated to pairs \((D^{2j},S^{2j-1})\) (see \cite{bbcg15}, Theorem~7.5), and S.~Amelotte for emphasizing the need to assume compatible characteristic data in the reduction to quasitoric and quoric rigidity.

\section{Moment--angle manifolds}
\label{sec:moment_angle}

Moment--angle manifolds arise naturally in algebraic and toric topology as spaces associated with simplicial complexes and simple polytopes. They were introduced by Davis and Januszkiewicz \cite{dj} and later studied extensively in \cite{bp} as a disc-circle decomposition of the Davis-Januszkiewicz universal space. In this section we recall only the material needed later in the paper. Following the natural description in terms of polyhedral products (see for instance \cite{bbcg}), we use the disk--sphere model for complex (quaternionic resp.) moment--angle manifolds.

Let \(\mathcal K\) be a simplicial complex on the vertex set \([m]=\{1,\ldots,m\}\), and let
\[
    (\mathbf X,\mathbf A)=\bigl((X_i,A_i)\bigr)_{i\in[m]}
\]
be a sequence of pairs of spaces with \(A_i\subseteq X_i\). The associated \emph{polyhedral product} is the subspace of the product \(\prod_{i=1}^m X_i\) defined by
\[
    \mathcal Z_{\mathcal K}(\mathbf X,\mathbf A)
    =
    \bigcup_{\sigma\in \mathcal K}
    (\mathbf X,\mathbf A)^\sigma
    \subseteq
    \prod_{i=1}^m X_i,
\]
where for each simplex \(\sigma\in \mathcal  K\)
\[
    (\mathbf X,\mathbf A)^\sigma
    =
    \prod_{i=1}^m Y_i,
    \qquad
    Y_i=
    \begin{cases}
    X_i, & i\in \sigma,\\
    A_i, & i\notin \sigma.
    \end{cases}
\]

When all pairs are equal to a fixed pair \((X,A)\), we write simply
\[
    (X,A)^{\mathcal K}
\]
for the corresponding polyhedral product. Thus the classical moment--angle complex associated to \(\mathcal K\) is
\[
    \mathcal Z_{\mathcal K}^{\mathbb C}
    =
    (D^2,S^1)^{\mathcal K},
\]
and its quaternionic analogue is
\[
    \mathcal Z_{\mathcal K}^{\mathbb H}
    =
    (D^4,S^3)^{\mathcal K}.
\]

For the purposes of this paper, we focus on the case where \(\mathcal K\) is the nerve of the facets of a simple polytope \(P\). In this case the above polyhedral products are moment--angle manifolds. The central point for our later arguments is that these manifolds can be viewed as total spaces of principal bundles over quasitoric and quoric manifolds, respectively. The moment--angle manifold itself, however, depends only on the combinatorics of \(P\); the choice of characteristic data enters only when one passes to a quasitoric or quoric quotient.

\subsection{The complex case}
\subsubsection{Complex moment--angle manifolds}
\label{sec:moment_angle_C}

Let \(\mathcal K\) be an abstract simplicial complex on the vertex set \([m]=\{1,\ldots,m\}\), that is, a collection of subsets of \([m]\) which is closed under taking subsets and contains the empty set. Following \cite{bp00}, for \(I\subset [m]\), define
\[
    (D^2,S^1)^I
    =
    \prod_{i=1}^{m}Y_i,
    \qquad
    Y_i=
    \begin{cases}
    D^2, & i\in I,\\
    S^1, & i\notin I.
    \end{cases}
\]
The \emph{(complex) moment--angle complex} associated to \(K\) is the polyhedral product
\[
    \mathcal Z_{\mathcal K}=(D^2,S^1)^{\mathcal K}
    =
    \bigcup_{I\in \mathcal K}(D^2,S^1)^I
    \subset (D^2)^m.
\]
Equivalently,
\[
    \mathcal Z_{\mathcal K}
    =
    \left\{
    (z_1,\ldots,z_m)\in (D^2)^m
    :
    \{i: |z_i|<1\}\in \mathcal K
    \right\}.
\]
The torus \(T^m=(S^1)^m\) acts coordinatewise on \((D^2)^m\), and \(\mathcal Z_{\mathcal K}\) is invariant under this action.

We now specialize to the case relevant for this paper. Let \(P^n\) be a simple convex polytope with facets
\[
    \mathcal F(P)=\{F_1,\ldots,F_m\}.
\]
Let \(K_P\) be the nerve of the covering of \(\partial P\) by the facets of \(P\). Thus
\[
    \{i_1,\ldots,i_k\}\in K_P
    \quad\Longleftrightarrow\quad
    F_{i_1}\cap\cdots\cap F_{i_k}\neq\varnothing .
\]
Equivalently, \(K_P\) is the boundary complex of the simplicial polytope dual to \(P\). The associated \emph{complex moment--angle manifold} is
\[
    \mathcal Z_P:=\mathcal Z_{K_P}=(D^2,S^1)^{K_P}.
\]
According to \cite[Theorem 4.1.4]{tt}, it is a closed smooth manifold of dimension \(m+n\). The canonical coordinatewise \(T^m\)-action on \(\mathcal Z_P\) has orbit space naturally homeomorphic to \(P\):
\[
    \mathcal Z_P/T^m\cong P.
\]

More explicitly, for \(x\in P\), let \(F(x)\) denote the unique face of \(P\) whose relative interior contains \(x\). If
\[
    F(x)=F_{i_1}\cap\cdots\cap F_{i_k},
\]
we define the coordinate subtorus
\[
    T_{F(x)}
    =
    \{(t_1,\ldots,t_m)\in T^m:
    t_j=1 \text{ for } j\notin \{i_1,\ldots,i_k\}\}.
\]
Then \(\mathcal Z_P\) admits the quotient description
\[
    \mathcal Z_P\cong (T^m\times P)/{\sim},
\]
where
\[
    (t,x)\sim(t',x')
    \quad\Longleftrightarrow\quad
    x=x'
    \text{ and }
    t^{-1}t'\in T_{F(x)}.
\]
Under this description, the \(T^m\)-action is induced by left multiplication on the first factor, and the orbit space is \(P\).

\subsubsection{Quasitoric manifolds}
\label{sec:quasitoric_manifolds}

Let \(T^n=(S^1)^n\). In the consequent definitions we follow \cite{dj}, taking into account some adjustments from \cite{tt}. The standard action of \(T^n\) on \(\mathbb C^n\) is the coordinatewise action
\[
    (t_1,\ldots,t_n)\cdot (z_1,\ldots,z_n)
    =
    (t_1z_1,\ldots,t_nz_n).
\]
The orbit space of this action is naturally identified with the positive cone \(\mathbb R_{\ge 0}^{n}\).

We first recall the notion of weak equivariance (see \cite{weimeler12}). Let \(T\) be a torus. Two \(T\)-spaces \(X\) and \(Y\) are called \emph{weakly \(T\)-equivariantly homeomorphic} if there exist an automorphism\(\theta:T\longrightarrow T\) and a homeomorphism \(f:X\longrightarrow Y\) such that \(f(t\cdot x)=\theta(t)\cdot f(x)\) for all \(t\in T\) and \(x\in X\).

A \(T^n\)-action on a closed topological \(2n\)-manifold \(M\) is called \emph{locally standard} if every point of \(M\) has an invariant neighbourhood which is weakly \(T^n\)-equivariantly homeomorphic to an invariant open subset of the standard \(T^n\)-representation on \(\mathbb C^n\). Thus the local models are the standard coordinatewise action of \(T^n\) on \(\mathbb C^n\), up to an automorphism of the acting torus.

The orbit space of a locally standard \(T^n\)-manifold is naturally a manifold with corners. 

\begin{definition}
A \emph{quasitoric manifold} is a $2n$–dimensional manifold $M$ equipped with a locally standard $T^n$–action whose orbit space is a simple polytope $P$.
\end{definition}

\medskip

Here, the orbit map
\[
\pi : M \longrightarrow P
\]
maps each $k$--dimensional orbit to a point in the interior of a codimension--$k$ face of $P$, for $k=0,\dots,n$.

\medskip

\begin{example}
All complex projective spaces $\mathbb CP^n$ and their equivariant connected sums and products are quasitoric manifolds.
\end{example}

The isotropy data of a quasitoric manifold are encoded by a characteristic function. Let \(\mathcal F(P)=\{F_1,\ldots,F_m\}\) be the set of facets of \(P\). A \emph{characteristic function} is a map \(\lambda:\mathcal F(P)\longrightarrow \mathbb Z^n\) such that each \(\lambda(F_i)\) is primitive and the following nonsingularity condition holds: whenever
\[
    F_{i_1}\cap\cdots\cap F_{i_k}\neq\varnothing,
\]
the vectors
\[
    \lambda(F_{i_1}),\ldots,\lambda(F_{i_k})
\]
span a rank \(k\) direct summand of \(\mathbb Z^n\).

The pair \((P,\lambda)\) determines a quasitoric manifold, denoted \(M(P,\lambda)\) (see \cite[Proposition~1.8]{dj}). Conversely, every quasitoric manifold over \(P\) arises from such characteristic data, up to the usual equivalences. Equivalently, one may describe \(M(P,\lambda)\) by the canonical model
\[
    M(P,\lambda)
    =
    (T^n\times P)/{\sim_\lambda},
\]
where
\[
    (s,x)\sim_\lambda(s',x')
    \quad\Longleftrightarrow\quad
    x=x'
    \text{ and }
    s^{-1}s'\in T^{\lambda}_{F(x)}.
\]
Here \(T^{\lambda}_{F(x)}\subset T^n\) is the subtorus generated by the circle subgroups corresponding to the facets containing \(F(x)\).

\subsubsection{Moment--angle manifolds as principal torus bundles}
\label{sec:ma_principal_bundle}

The relation between moment--angle manifolds and quasitoric manifolds fits naturally into the general description of locally standard torus actions.  In Yoshida's formulation \cite{Yoshida}, a locally standard \(T^n\)-manifold with orbit space \(Q\) can be described as a quotient of a principal \(T^n\)-bundle over \(Q\) by an equivalence relation determined by the characteristic function of the action. In the special case where \(Q=P\) is a simple polytope, this description recovers the usual canonical model of a quasitoric manifold. From the moment--angle point of view, one starts instead from the larger \(T^m\)-space \(\mathcal Z_P\), whose orbit space is \(P\). This is not Yoshida's principal \(T^n\)-bundle; rather, it carries the full coordinatewise \(T^m\)-symmetry before a characteristic matrix has been chosen. Once a characteristic matrix is fixed, the kernel \(K_\Lambda\subset T^m\) of the corresponding characteristic homomorphism acts freely on \(\mathcal Z_P\), and the quotient
\[
    \mathcal Z_P/K_\Lambda
\]
is the associated quasitoric manifold. We describe this in detail.

Let \(P^n\) be a simple polytope with facets \(\mathcal F(P)=\{F_1,\ldots,F_m\}\), and let \(\lambda:\mathcal F(P)\longrightarrow \mathbb Z^n \) be a characteristic function.  We write
\[
    \Lambda=(\lambda_{ij})\in \mathbb Z^{n\times m}
\]
for the corresponding characteristic matrix, whose \(j\)-th column is the characteristic vector \(\lambda(F_j)\).  Equivalently, \(\Lambda\) is the homomorphism
\[
    \Lambda:\mathbb Z^m\longrightarrow \mathbb Z^n,
    \qquad
    e_j\longmapsto \lambda(F_j),
    \quad j=1,\ldots,m.
\]
Passing to tori gives a homomorphism
\[
    \Lambda:T^m\longrightarrow T^n
\]
defined by
\[
    \Lambda(t_1,\ldots,t_m)
    =
    \left(
    \prod_{j=1}^{m}t_j^{\lambda_{1j}},
    \ldots,
    \prod_{j=1}^{m}t_j^{\lambda_{nj}}
    \right),
    \qquad t_j\in S^1.
\]
Since \(\Lambda\) has rank \(n\), its kernel
\[
    K_\Lambda=\ker(\Lambda)
\]
is a subtorus of \(T^m\) of dimension \(m-n\).  Thus we obtain a short exact sequence of compact tori
\[
    1
    \longrightarrow
    K_\Lambda
    \longrightarrow
    T^m
    \xrightarrow{\ \Lambda\ }
    T^n
    \longrightarrow
    1,
\]
where \(K_\Lambda\cong T^{m-n}\) (for details see \cite[\S 4.8]{tt}, also compare \cite{lan}).

The nonsingularity condition on the characteristic matrix is equivalent to the condition that \(K_\Lambda\) intersects trivially the coordinate isotropy subgroups of the canonical \(T^m\)-action on \(\mathcal Z_P\). Hence \(K_\Lambda\) acts freely on \(\mathcal Z_P\).  Therefore the quotient
\[
    M(P,\Lambda)=\mathcal Z_P/K_\Lambda
\]
is a closed \(2n\)-manifold.  Moreover, the residual torus
\[
    T^m/K_\Lambda\cong T^n
\]
acts locally standardly on \(M(P,\Lambda)\), and the orbit space is naturally identified with \(P\).  Thus \(M(P,\Lambda)\) is the quasitoric manifold determined by the characteristic pair
\((P,\lambda)\). Thus $\mathcal Z_P$ can be viewed as the total space of a principal $T^{m-n}$--bundle over a quasitoric manifold (see \cite[Proposition 3.1.5]{bp00}).

Consequently, the quotient map
\[
    \pi_\Lambda:\mathcal Z_P\longrightarrow M(P,\Lambda)
\]
is a principal \(K_\Lambda\)-bundle
\[
    K_\Lambda
    \longrightarrow
    \mathcal Z_P
    \xrightarrow{\ \pi_\Lambda\ }
    M(P,\Lambda).
\]
This principal bundle is the fibration used throughout the rigidity arguments below.

\begin{example}[The classical Hopf fibration]
Let \(P=\Delta^n\) be the \(n\)-simplex.  Then \(P\) has \(m=n+1\) facets, say \(\mathcal F(\Delta^n)=\{F_1,\ldots,F_{n+1}\}\). The associated complex moment--angle manifold is \(\mathcal Z_{\Delta^n}\cong S^{2n+1}\). Indeed, the dual simplicial complex \(K_{\Delta^n}\) is the boundary of the \(n\)-simplex, and hence \(\mathcal Z_{\Delta^n}=(D^2,S^1)^{\partial\Delta^n}\cong S^{2n+1}\). The complex projective space \(\mathbb CP^n\) is the standard quasitoric manifold over \(\Delta^n\). Its orbit space under the standard \(T^n\)-action is \(\Delta^n\), and its characteristic matrix may be chosen as
\[
    \Lambda=
    \begin{pmatrix}
    1 & 0 & \cdots & 0 & -1\\
    0 & 1 & \cdots & 0 & -1\\
    \vdots & \vdots & \ddots & \vdots & \vdots\\
    0 & 0 & \cdots & 1 & -1
    \end{pmatrix}
    \in \mathbb Z^{n\times(n+1)}.
\]
Thus the corresponding homomorphism \(\Lambda:T^{n+1}\longrightarrow T^n\) is given by \(\Lambda(t_1,\ldots,t_{n+1}) =    (t_1t_{n+1}^{-1},\ldots,t_nt_{n+1}^{-1})\). Its kernel is the diagonal circle \(K_\Lambda =\{(t,\ldots,t)\in T^{n+1}:t\in S^1\}\cong S^1\). Therefore the quasitoric quotient associated with this characteristic matrix is
\[
    M(\Delta^n,\Lambda)
    =
    \mathcal Z_{\Delta^n}/K_\Lambda
    \cong
    \mathbb CP^n.
\]

Consequently, the principal torus bundle
\[
    K_\Lambda
    \longrightarrow
    \mathcal Z_{\Delta^n}
    \longrightarrow
    M(\Delta^n,\Lambda)
\]
becomes
\[
    S^1
    \longrightarrow
    S^{2n+1}
    \longrightarrow
    \mathbb CP^n,
\]
which is precisely the classical \emph{Hopf fibration}. In particular, for \(n=1\) we recover
\[
    S^1
    \longrightarrow
    S^3
    \longrightarrow
    \mathbb CP^1\cong S^2.
\]
\end{example}

\subsection{The quaternionic case}
\label{subsec:quaternionic_case}

We now discuss the quaternionic analogue of the preceding construction. Throughout, let $\mathbb H$ denote the space of quaternions, and let $Q:=S^3\subset \mathbb H$ be the compact non-abelian Lie group of unit quaternions. Thus $Q\cong \mathrm{SU}(2)\cong \mathrm{Sp}(1)$. For $n\ge 1$, we write $Q^n=(S^3)^n$ and refer to it as the \emph{quaternionic torus}.

\subsubsection{Quaternionic moment--angle manifolds}
\label{sec:moment_angle_H}

Let \(\mathcal K\) be a simplicial complex on the vertex set \([m]=\{1,\ldots,m\}\). Let \(D^4\subset \mathbb H\) be the closed unit ball in the quaternionic line, and let \(S^3=\partial D^4\) be the group of unit quaternions. Taking products, we obtain  
\[
(D^4)^m = \{(h_1,\dots,h_m)\in \mathbb H^m \mid |h_i|\le 1 \text{ for } i=1,\dots,m\}
\]
\noindent which will be called the \emph{quaternionic polydisc}, and the quaternionic tori corresponds to
\[
Q^m = (S^3)^m = (\partial D^4)^m.
\]
In this quaternionic setting, the natural analogue to the complex $\mathcal Z_{\mathcal K}$ is obtained from the sequence of pairs of spaces $(D^{4}, S^{3})$. The \emph{quaternionic moment--angle complex}
associated to \(\mathcal K\) is the polyhedral product
\[
    \mathcal Z_{\mathcal K}^{\mathbb H}
    =
    (D^4,S^3)^{\mathcal K}
    =
    \bigcup_{I\in \mathcal K}(D^4,S^3)^I
    \subset (D^4)^m,
\]
where, for \(I\subset[m]\),
\[
    (D^4,S^3)^I
    =
    \prod_{i=1}^{m}Y_i,
    \qquad
    Y_i=
    \begin{cases}
    D^4, & i\in I,\\
    S^3, & i\notin I.
    \end{cases}
\]
\noindent 
The (standard) coordinatewise left multiplication of $Q^m$ on $(D^4)^m\subset\mathbb H^m$ 
$$(q_1,\ldots,q_m)\cdot(h_1,\ldots,h_m) = (q_1h_1,\ldots,q_mh_m)$$
restricts to a canonical action on $\mathcal Z_{\mathcal K}^{\mathbb H}$. Indeed, for every simplex $\sigma\in\mathcal K$ the subspace
\[
\prod_{i\in \sigma}D^4
\times
\prod_{j\notin \sigma}S^3
\]
is preserved by the action since left multiplication by a unit quaternion preserves both the unit disc $D^4$ and its boundary $S^3$. Since $\mathcal Z_{\mathcal K}^{\mathbb H}$ is the union of these invariant subspaces, it follows that $\mathcal Z_{\mathcal K}^{\mathbb H} \subseteq (D^{4})^{m}$ is $Q^m$--invariant.

Therefore, this action preserves the quaternionic polydisc \((D^4)^m\) and its boundary factors \(S^3=\partial D^4\). The coordinatewise \(Q^m\)-action on \((D^4)^m\) restricts to an action on \(\mathcal Z_{\mathcal K}^{\mathbb H}\).

\begin{example}[The simplex]\label{ex:simplex}
Let \(\mathcal K=\partial\Delta^n\). Since every proper subset of \([n+1]\) is a simplex of \(\mathcal K\),
\[
\mathcal Z_{\mathcal K}^{\mathbb H}=\bigcup_{i=1}^{n+1} \Bigl( (D^4)^{i-1}\times S^3\times (D^4)^{n+1-i} \Bigr).
\]
This is precisely the boundary of the quaternionic polydisc
\[
(D^4)^{n+1},
\]
and therefore
\[
\mathcal Z_{\mathcal K}^{\mathbb H}=\partial (D^4)^{n+1} \cong S^{4n+3}.
\]
In particular, the quaternionic moment--angle manifold associated to the simplex \(\Delta^n\) is the sphere \(S^{4n+3}\).
\end{example}
\medskip

Let \(P^n\) be a simple polytope with facets \(\mathcal F(P)=\{F_1,\ldots,F_m\}\). Let \(K_P\) be the nerve of the covering of \(\partial P\) by the facets of \(P\), as before. The associated \emph{quaternionic moment--angle manifold} is
\[
    \mathcal Z_P^{\mathbb H}
    :=
    \mathcal Z_{K_P}^{\mathbb H}
    =
    (D^4,S^3)^{K_P}.
\]

Let
\[
P^n=\{x\in\mathbb R^n : A_Px+b_P\ge 0\}
\]
be a simple polytope with \(m\) facets. The affine map
\[
i_P:\mathbb R^n\longrightarrow\mathbb R^m, \qquad i_P(x)=A_Px+b_P,
\]
restricts to an embedding
\[
i_P:P^n\hookrightarrow \mathbb R_{\ge0}^m.
\]
The \(i\)-th coordinate of \(i_P(x)\) vanishes precisely when \(x\in F_i\) (for details see \cite{panov2010}). 

As in the classical (complex) case, $P^n$ admits a natural structure of a cubical complex embedded in the subset of faces of the $m$--cube $I^m$ (for details see \cite[Theorem 2.2]{bp} and \cite[pp.~77]{ho}). More precisely, there exists a canonical embedding
\[
i_P : P^n \hookrightarrow I^m
\]
such that if
\[
v = F_{i_1}\cap\cdots\cap F_{i_n}
\]
is a vertex of $P^n$, then a neighbourhood cube of $v$ maps onto the $n$–face of $I^m$ defined by
\[
x_j=1 \quad \text{for } j\notin\{i_1,\dots,i_n\}.
\]

\begin{lemma}\label{lem:embedding cover}
Let \(P^n\) be a simple polytope with \(m\) facets. Then the cubical embedding
\[
i_P:P^n\hookrightarrow I^m
\]
is covered by a \(Q^m\)-equivariant embedding
\[
i_e:\mathcal Z_P^{\mathbb H}\hookrightarrow (D^4)^m.
\]
\end{lemma}

\begin{proof}
The proof is similar to \cite[Lemma~6.6]{bp}. For a vertex \(v\in P^n\), let \(U_v\subset P^n\) be the open set obtained by deleting all faces not containing \(v\). The sets \(U_v\) form an atlas of the manifold with corners \(P^n\) (see \cite[Construction~5.8]{bp}).

Let \(B_v=\pi^{-1}(U_v)\subset \mathcal Z_P^{\mathbb H}\). If \(v=F_{i_1}\cap\cdots\cap F_{i_n}\), then 
\[B_v\cong (D^4)^n\times (S^3)^{m-n}\] 
because by the disk--sphere description of \(\mathcal Z_P^{\mathbb H}\), the coordinates corresponding to the facets meeting at \(v\) vary in \(D^4\), whereas the remaining coordinates lie in \(S^3\). 

The obvious inclusion
\[
(D^4)^n\times (S^3)^{m-n}
\hookrightarrow
(D^4)^m
\]
is \(Q^m\)-equivariant. These local embeddings are compatible on overlaps, hence they glue to a global \(Q^m\)-equivariant embedding
\[
i_e:\mathcal Z_P^{\mathbb H}\hookrightarrow (D^4)^m.
\]
\end{proof}

The following result shows that, in contrast to the case of an arbitrary simplicial complex, the space $\mathcal Z_P^{\mathbb H}$ is a smooth manifold.

\begin{thm}\label{thm:quaternionic-ZP-smooth}
For any simple polytope \(P^n\) with \(m\) facets, \(\mathcal Z_P^{\mathbb H}\) is a smooth \((3m+n)\)-dimensional manifold and the canonical \(Q^m\)-action is smooth.
\end{thm}

\begin{proof}
Let
\[
\rho:\mathcal Z_P^{\mathbb H}\longrightarrow P^n
\]
denote the projection induced by the polyhedral-product structure. For every vertex \(v\in P^n\), let \(U_v\subset P^n\) be the open set obtained by deleting all faces not containing \(v\). The sets \(U_v\) form an atlas of the manifold with corners \(P^n\). Suppose
\[
v=F_{i_1}\cap\cdots\cap F_{i_n}.
\]
Since \(P^n\) is simple, these are precisely the \(n\) facets meeting at \(v\). By the disk--sphere description of \(\mathcal Z_P^{\mathbb H}\),
\[
B_v:=\rho^{-1}(U_v) \cong (D^4)^n\times (S^3)^{m-n}.
\]

Let \(I_v\subset[m]\) denote the set of facets containing \(v\). Via the \(Q^m\)-equivariant embedding
\[
i_e:\mathcal Z_P^{\mathbb H}\hookrightarrow (D^4)^m
\]
of Lemma~\ref{lem:embedding cover}, the set \(B_v\) is identified with the face
\[
C_{I_v} = \{h\in (D^4)^m : h_j\in S^3 \text{ for } j\notin I_v\}
\]
of the manifold with corners \((D^4)^m\). Hence for any two vertices \(v,w\),
\[
B_v\cap B_w = C_{I_v}\cap C_{I_w} = C_{I_v\cap I_w}.
\]

Therefore all charts \(B_v\) are realized as faces of the same ambient manifold with corners \((D^4)^m\). On overlaps, the coordinate descriptions are induced from the ambient coordinates of \(\mathbb H^m\), so the transition maps are simply restrictions of the identity map on \(\mathbb H^m\). In particular, they are smooth.

Now let \(x\in \mathcal Z_P^{\mathbb H}\) project to the relative interior of a face \(F_\tau\subset P^n\) of codimension \(k\). Since \(P^n\) is simple, exactly \(k\) facets contain \(F_\tau\). Locally near the image of \(x\) in \(P^n\), the polytope is modelled on \(\mathbb R^{n-k}\times \mathbb R_{\geq 0}^{k}\). Under the quaternionic disk--sphere construction, a neighbourhood of \(x\) in \(\mathcal Z_P^{\mathbb H}\) is therefore modelled on \(\mathbb R^{n-k}\times (D^4)^k\times Q^{m-k}\). Here the factor \(\mathbb R^{n-k}\) corresponds to directions tangent to the face \(F_\tau\), the factors \((D^4)^k\) correspond to the \(k\) facets containing \(F_\tau\), and the factor \(Q^{m-k}\) corresponds to the remaining coordinates, which lie in \(S^3=Q\).

Its dimension is
\[(n-k)+4k+3(m-k)=3m+n,\]
which is independent of \(k\). Hence \(\mathcal Z_P^{\mathbb H}\) is a smooth manifold of dimension \(3m+n\).

The $Q^m$–action is smooth in these local charts, hence globally smooth. This completes the proof.
\end{proof}

The canonical \(Q^m\)-action is smooth, and its orbit space \(\mathcal Z_P^{\mathbb H}/Q^m\) is naturally homeomorphic to \(P\). More explicitly, for \(x\in P\), let \(F(x)\) be the unique face whose
relative interior contains \(x\), and let
\[
    I_x=\{i\in[m]:x\in F_i\}
\]
be the set of facets containing \(x\).  Define the coordinate
quaternionic subtorus
\[
    Q^{I_x}
    =
    \{(q_1,\ldots,q_m)\in Q^m:
    q_j=1 \text{ for } j\notin I_x\}.
\]

Consider the map \(\mathbb R_{\ge0}\times Q \longrightarrow \mathbb H\) with \((r,q)\longmapsto rq\). Taking products yields \(\mathbb R_{\ge0}^m\times Q^m \longrightarrow \mathbb H^m\). For
\(h=(h_1,\ldots,h_m)\in \mathbb H^m\), let \(r_i=|h_i|\), and \(I(h)=\{i\in[m]:h_i=0\}\). The fibre over $h$ is
\[
(r_1,\ldots,r_m)\times Q^{|I(h)|}.
\]
\noindent
Therefore
\[
\mathbb H^m
\cong
\mathbb R_{\ge0}^m\times Q^m/\!\sim,
\]
where the equivalence is defined by
\[
(r,q_1)\sim(r,q_2)
\quad\Longleftrightarrow\quad
q_1^{-1}q_2\in Q^{I(r)},
\]
and
\[
I(r)=\{i:r_i=0\}.
\]
This identification reconstructs $\mathbb H^m$ from its orbit space $\mathbb R_{\ge0}^m$ and the $Q^m$-action; for details see \cite[\S3.2]{ho}.

Let \(P^n\) be a simple polytope with facets \(F_1,\ldots,F_m\). For \(p\in P^n\), define \(I_p = \{i\in[m]:p\in F_i\}\). The reconstruction of \(\mathbb H^m\) described above restricts to the embedded polytope \(P^n\) and yields the following quaternionic analogue of the classical identification-space model for moment--angle manifolds.

\begin{prop}
The quaternionic moment--angle manifold \(\mathcal Z_P^{\mathbb H}\) is
\(Q^m\)-equivariantly homeomorphic to the quotient
\[
P^n\times Q^m/\!\sim ,
\]
where
\[
(p,q_1)\sim(p,q_2)
\qquad\Longleftrightarrow\qquad
q_1^{-1}q_2\in Q^{I_p}.
\]
\end{prop}

\begin{proof}
By \cite[\S 3.2]{ho} the orbit map
\[
\mu_{\mathbb H}:\mathbb H^m\to \mathbb R_{\ge0}^m,
\qquad
(h_1,\ldots,h_m)
\longmapsto
(|h_1|^2,\ldots,|h_m|^2),
\]
identifies $\mathbb H^m$ with the quotient space
\[
\mathbb R_{\ge0}^m\times Q^m/\!\sim,
\]
where the isotropy subgroup over a point $y=(y_1,\ldots,y_m)$ is precisely $Q^{I(y)}$ with $I(y)=\{i:y_i=0\}$. Now we show that  
\[
\mathcal Z_P^{\mathbb H}
=
\mu_{\mathbb H}^{-1}(i_P(P^n)).
\]

\noindent
Indeed, by Lemma \ref{lem:embedding cover}, we may regard
\(\mathcal Z_P^{\mathbb H}\) as a \(Q^m\)-invariant subspace of \((D^4)^m\).

Recall that the orbit map
\[
\mu_{\mathbb H}(h_1,\ldots,h_m)
=
(|h_1|^2,\ldots,|h_m|^2)
\]
identifies the orbit space
\[
(D^4)^m/Q^m\cong I^m.
\]

Let $h=(h_1,\ldots,h_m)\in (D^4)^m$. The coordinates \(h_i\) that vanish determine a subset $I(h)=\{i\in[m]:h_i=0\}$. Under the cubical embedding (Lemma~\ref{lem:embedding cover}) $i_P:P^n\hookrightarrow I^m$, the face of \(P^n\) containing a point \(x\) is encoded by the set of vanishing coordinates of \(i_P(x)\). More precisely,
\[
i_P(x)_i=0
\qquad\Longleftrightarrow\qquad
x\in F_i.
\]

By the disk--sphere model,
\[
\mathcal Z_P^{\mathbb H}
=
\bigcup_{\sigma\in K}
\left(
\prod_{i\in\sigma}D^4
\times
\prod_{j\notin\sigma}S^3
\right),
\]
so a point \(h\in(D^4)^m\) belongs to \(\mathcal Z_P^{\mathbb H}\) precisely when the set \(I(h)=\{i:h_i=0\}\) coincides with the set of facets containing some face of \(P^n\).

Equivalently,
\[
h\in\mathcal Z_P^{\mathbb H}
\quad\Longleftrightarrow\quad
\mu_{\mathbb H}(h)\in i_P(P^n).
\]

Hence
\[
\mathcal Z_P^{\mathbb H}
=
\mu_{\mathbb H}^{-1}(i_P(P^n))
\]
which is the pullback of \(\mu_{\mathbb H}\) along \(i_P\). The construction is evidently \(Q^m\)-equivariant.

Restricting the above quotient description of $\mathbb H^m$ to $\mu_{\mathbb H}^{-1}(i_P(P^n))$ gives $$\mathcal Z_P^{\mathbb H} \cong P^n\times Q^m/\!\sim.$$

Finally, the coordinates of the affine embedding $i_P:P^n\hookrightarrow \mathbb R_{\ge0}^m$ vanish precisely on the facets containing a point. Hence $I(i_P(p))=I_p$, and therefore
\[
(p,q_1)\sim(p,q_2)
\quad\Longleftrightarrow\quad
q_1^{-1}q_2\in Q^{I_p}.
\]
\end{proof}

\begin{remark} The above construction can be considered as an analogue of the identification space construction, which goes back to the work of Vinberg \cite{Vinberg} on Coxeter groups and was presented in the above form in the work of Davis and Januszkiewicz \cite{dj}. 
\end{remark}

\subsubsection{Quoric manifolds}
\label{subsec:quoric_manifolds}

We now introduce the quaternionic analogue of quasitoric manifolds \cite{dj}, namely \emph{quoric manifolds}. This subsection is a brief summary of \cite{ho}. There is a natural action of $Q^n$ on $\mathbb H^n$ given by coordinatewise left multiplication,
\[
((s_1,\ldots,s_n),(h_1,\ldots,h_n))
\longmapsto
(s_1h_1,\ldots,s_nh_n).
\]
We refer to this as the \emph{standard $Q^n$-action} on $\mathbb H^n$.

\medskip

Such actions have orbit space \(\mathbb R_{\ge0}^n\), and their orbit types are indexed by the faces of the positive cone. In detail, conjugacy classes are naturally associated with the faces of the orbits. We refer to these classes as \emph{isotropy classes}.

\medskip
In \cite{ho}, Hopkinson considers general actions given by coordinatewise conjugation, which generalize the standard coordinatewise multiplication action. Let $Q^n$ act on $\HH^n$,
\begin{equation} \label{eq:general_action}
((s_1,...,s_n),(h_1,...,h_n))\mapsto (s_{l_1}h_1s_{r_1}^{-1},...,s_{l_n}h_ns_{r_n}^{-1})
\end{equation} 
where each subscript is the index of some coordinate subgroup $Q_k$ (i.e. $Q_k=1\times\cdots\times 1\times Q\times 1\times\cdots 1$, with $Q$ in the k-th position) of $Q^n$ or is the empty set, with $s_{\emptyset}:=1$.

These general actions include the standard coordinatewise left multiplication action as a special case. The choice \(l_i=i\) and \(r_i=\emptyset\) for all $i$ recovers the coordinatewise action
\[
(s_1,\ldots,s_n)\cdot(h_1,\ldots,h_n)=(s_1h_1,\ldots,s_nh_n).
\]

Hopkinson observed that the essential feature is not the specific coordinatewise multiplication but the correspondence between faces of the orbit space and isotropy classes. This leads to the notion of an isotropy functor. Associated to any action of the form \eqref{eq:general_action} is an assignment
\[
\ell: \{\text{faces of }\mathbb R_{\ge0}^n\} \to \{\text{isotropy classes of subgroups of }Q^n\},
\]
called the \emph{isotropy functor}.

\begin{definition}\label{def:acceptability}
An isotropy functor \(\ell\) associated to an action of the form
\eqref{eq:general_action} is called \emph{acceptable} if \(\ell\) is
injective.
\end{definition}

\begin{remark}\label{rmk:acceptability_rank_incidence}
Although Definition~\ref{def:acceptability} is stated in terms of injectivity, the functor \(\ell\) is not an arbitrary map from the face poset of \(\mathbb R_{\ge0}^{n}\) to isotropy classes. It is the isotropy functor arising from an action of the form \eqref{eq:general_action}. For such actions, the rank and incidence behaviour of the isotropy classes is already built into the local model: points lying over the relative interior of a codimension-\(k\) face have isotropy of rank \(k\), and inclusions of faces correspond to inclusions of the associated isotropy classes. Thus these rank and incidence conditions are not extra hypotheses in the definition of acceptability; they are consequences of the fact that \(\ell\) comes from the orbit-type stratification of a quaternionic corner.

The additional condition imposed by acceptability is precisely injectivity.  Without injectivity, two distinct faces could determine the same isotropy class, so the face stratification of the orbit space would not be faithfully recorded by isotropy data.  Acceptability rules out this degeneracy.  This is the sense in which acceptable isotropy functors give the quaternionic analogue of the nonsingularity condition for characteristic data in the classical toric setting; see \cite[Proposition~3.1.23]{ho} and the equivalent formulations in \cite[p.~60]{ho}.
\end{remark}

\begin{remark}
In the above Definition \ref{def:acceptability}, injectivity alone is not enough to discuss the relation between quoric manifolds and the quaternionic moment--angle constructions. We will introduce the notion of global characteristic functions, following \cite[\S 5]{ho} in the next subsection.
\end{remark}

\begin{definition}
An action of the form \eqref{eq:general_action} whose isotropy functor is acceptable is called a \emph{regular quaternionic action}.
\end{definition}

\begin{remark}
Regular actions retain the same orbit-space structure and orbit-type stratification, as the standard actions (see \cite[Proposition 3.1.23]{ho}). In particular, their orbit space is naturally identified with \(\mathbb R_{\ge0}^n\), and orbit types are indexed by the faces of \(\mathbb R_{\ge0}^n\).
\end{remark}

\begin{definition}\label{def:regular_corner}
The space \(\mathbb H^n\) equipped with a regular quaternionic action is called a \emph{regular \(Q^n\)-corner}.
\end{definition}

Let $M^{4n}$ be a topological manifold equipped with a continuous $Q^n$-action.

\begin{definition}\label{def:loc_reg}
An action of \(Q^{n}\) on a  \(4n\)-manifold \(M\) is called \emph{locally regular} if every point \(x\in M\) admits a \(Q^{n}\)-invariant open neighbourhood \(U_x\subseteq M\), an open \(Q^{n}\)-invariant subset \(V\) of a regular \(Q^n\)-corner, and an equivariant homeomorphism $\varphi:U_x\longrightarrow V$.
\end{definition}

We now define quoric manifolds.

\medskip
\begin{definition}\label{def:quoric}
Let $P^n$ be a simple polytope. A \emph{quoric manifold} over $P^n$ is a $4n$-dimensional manifold $M^{4n}$ equipped with a $Q^n$-action such that:
\begin{itemize}
    \item[(1)] the action of $Q^n$ on $M^{4n}$ is locally regular;
    \item[(2)] there is a continuous projection
    \[
    \pi:M^{4n}\to P^n
    \]
    whose fibres are precisely the $Q^n$-orbits.
\end{itemize}
\end{definition}

Thus the orbit space is locally modelled on \(\mathbb R_{\ge0}^n\), and hence inherits the structure of a manifold with corners. By \cite[Definition~4.2.3]{ho}, the orbit-type data of a quoric manifold are encoded by a characteristic function
\[
    \lambda:
    \{\text{faces of }P\}
    \longrightarrow
    \{\text{isotropy classes of subgroups of }Q^n\}.
\]
The pair \((P,\lambda)\) is called a \emph{characteristic pair}. Given such data, one forms the \emph{canonical model}
\[
    M(P,\lambda)
    =
    (Q^n\times P)/{\sim_\lambda}.
\]
For \(x\in P\), let \(F(x)\) be the unique face whose relative interior contains \(x\), and choose a representative subgroup
\[
    \widehat{\lambda}(F(x))\subset Q^n
\]
of the isotropy class assigned to \(F(x)\).  The equivalence relation is
defined by
\[ (q,x)\sim_\lambda(q',x') \quad\Longleftrightarrow\quad x=x'
    \text{ and }
    q^{-1}q'\in \widehat{\lambda}(F(x)).
\]
The group \(Q^n\) acts on \(M(P,\lambda)\) by left multiplication on the first factor, and the orbit space is naturally identified with \(P\).

The canonical model construction exhausts all quoric manifolds (see \cite[Proposition 4.2.10]{ho}).

\begin{example}
The quaternionic projective space \(\mathbb H P^n\), equipped with its standard \(Q^n\)-action, is a quoric manifold over the simplex \(\Delta^n\).
\end{example}

\subsubsection{Quaternionic moment--angle manifolds as principal quaternionic torus-bundles}
\label{sec:ma_principal_Q_bundle}

We now explain the principal-bundle interpretation of the relationship between quaternionic moment--angle manifolds and quoric manifolds. This is the quaternionic analogue of the Yoshida-type viewpoint used in the complex case (see details in \cite{bg26}): one starts with a space carrying a large group action and then obtains the locally regular quotient by dividing out a freely acting kernel subgroup.

In the quaternionic setting, however, an additional point must be taken into account. The acting group $Q^m=(S^3)^m$ is non-abelian, so the passage from characteristic data to a group homomorphism
\[
    Q^m\longrightarrow Q^n
\] is not formal (see \cite[\S 5]{ho}).  Thus, unlike the complex case, a characteristic matrix does not automatically define a homomorphism of acting groups. The existence of such a homomorphism, and hence of a principal \(Q^{m-n}\)-bundle
\[
    Q^{m-n}\longrightarrow \mathcal Z_P^{\mathbb H}
    \longrightarrow M,
\]
requires additional compatibility conditions on the quoric characteristic data.

The required compatibility is the \emph{globality} condition introduced by Hopkinson. It is precisely this condition that allows the characteristic data of a quoric manifold to be lifted to the canonical \(Q^m\)-action on the quaternionic moment--angle manifold.

\begin{definition}[Global characteristic function {\cite[Definition~5.1.1]{ho}}]
A characteristic function
\[
\lambda:\{\emph{faces of }P^n\}\longrightarrow
\{\emph{isotropy classes of subgroups of }Q^n\}
\]
is called \emph{global} if, for every pair of facets \(F_a,F_b\subset P^n\) (not necessarily meeting at a common vertex), the isotropy classes \(\lambda(F_a)\) and \(\lambda(F_b)\) are compatible.
\end{definition}

\begin{remark}
Compatibility in the above Definition means that, after choosing representatives of the corresponding conjugacy classes, the associated coordinate quaternionic subtori may be simultaneously realized inside a common quaternionic torus \(Q^m\); equivalently, their indexing subsets are either disjoint or one is contained in the other (cf.~\cite[Definition~2.1.14]{ho}).
\end{remark}

Under this hypothesis, in \cite[Corollary~5.1.4]{ho} Hopkinson shows that the characteristic data can be realized by a homomorphism
\[
\Lambda:Q^m\longrightarrow Q^n
\]
whose kernel is a quaternionic subtorus
\[
K_\Lambda=\ker(\Lambda)\cong Q^{m-n}.
\]

This provides precisely the additional structure needed to relate quaternionic moment–angle manifolds and quoric manifolds.

\begin{prop}[Quaternionic moment--angle/quoric quotient]
\label{prop:ma_quoric_quotient}
Let \(P^n\) be a simple polytope with \(m\) facets, and let \(M^{4n}\) be a quoric manifold over \(P\) whose characteristic function \(\lambda\) is global. Then there exists a surjective homomorphism $\Lambda:Q^m\longrightarrow Q^n$ with kernel $K_\Lambda=\ker(\Lambda)\cong Q^{m-n}$, such that \(K_\Lambda\) acts freely on \(\mathcal Z_P^{\mathbb H}\) and \(\mathcal Z_P^{\mathbb H}/K_\Lambda \cong M^{4n}\). Equivalently, the quotient map is a principal \(K_\Lambda\)-bundle
\[
    K_\Lambda
    \longrightarrow
    \mathcal Z_P^{\mathbb H}
    \longrightarrow
    M^{4n}.
\]
\end{prop}

\begin{proof}
By Hopkinson's classification theorem for quoric manifolds \cite[Proposition~4.2.10]{ho}, the quoric manifold \(M^{4n}\) is \(Q^n\)-equivariantly homeomorphic to the canonical model
\[
    \mathcal M(P,\lambda)
    =
    (Q^n\times P)/{\sim_\lambda}.
\]

Since the characteristic function \(\lambda\) is global, \cite[Corollary~5.1.4]{ho} implies that the characteristic data are induced by an acceptable isotropy functor on the set of facets of \(P\). Equivalently, there is a surjective homomorphism
\[
    \Lambda:Q^m\longrightarrow Q^n
\]
whose kernel is a quaternionic subtorus
\[
    K_\Lambda=\ker(\Lambda)\cong Q^{m-n}.
\]

By \cite[Propositions~5.2.4 and~5.2.6]{ho} we have that \(K_\Lambda\) acts freely on the quaternionic moment--angle manifold \(\mathcal Z_P^{\mathbb H}\), and that the quotient is naturally \(Q^n\)-equivariantly homeomorphic to the canonical model:
\[
    \mathcal Z_P^{\mathbb H}/K_\Lambda
    \cong
    \mathcal M(P,\lambda).
\]
Combining this with the classification homeomorphism
\[
    \mathcal M(P,\lambda)\cong M^{4n}
\]
gives
\[
    \mathcal Z_P^{\mathbb H}/K_\Lambda
    \cong
    M^{4n}.
\]
Since \(K_\Lambda\) acts freely and is a compact Lie group, the quotient map
\[
    \mathcal Z_P^{\mathbb H}
    \longrightarrow
    \mathcal Z_P^{\mathbb H}/K_\Lambda
\]
is a principal \(K_\Lambda\)-bundle.  Therefore
\[
    K_\Lambda
    \longrightarrow
    \mathcal Z_P^{\mathbb H}
    \longrightarrow
    M^{4n}
\]
is a principal \(Q^{m-n}\)-bundle.
\end{proof}

\begin{remark}
Globality is an essential hypothesis. Indeed, Hopkinson proves \cite[Proposition~5.2.7]{ho} that a quoric manifold arises as the quotient of a quaternionic moment--angle manifold by a free quaternionic subtorus if and only if its characteristic function is global.
\end{remark}

\begin{example}[The quaternionic Hopf fibration]
Let \(P=\Delta^n\) be the \(n\)-simplex.  Then \(P\) has \(m=n+1\) facets, and the associated quaternionic moment--angle manifold is
\[
    \mathcal Z_{\Delta^n}^{\mathbb H}
    \cong
    S^{4n+3}.
\]
Indeed,
\[
    K_{\Delta^n}=\partial\Delta^n,
\]
and therefore
\[
    \mathcal Z_{\Delta^n}^{\mathbb H}
    =
    (D^4,S^3)^{\partial\Delta^n}
    =
    \partial(D^4)^{n+1}
    \cong
    S^{4n+3}.
\]

The quaternionic projective space \(\mathbb H P^n\) is the standard quoric manifold over \(\Delta^n\). Since \(m-n=1\), the kernel of the corresponding characteristic homomorphism is
\[
    K_\Lambda\cong Q=S^3.
\]
Thus the principal bundle above becomes
\[
    S^3
    \longrightarrow
    S^{4n+3}
    \longrightarrow
    \mathbb H P^n,
\]
which is precisely the \emph{quaternionic Hopf fibration}.

In particular, for \(n=1\), we recover
\[
    S^3
    \longrightarrow
    S^7
    \longrightarrow
    \mathbb H P^1\cong S^4.
\]
\end{example}

\section{Reduction to rigidity of quasitoric and quoric manifolds}
\label{sec:reduction}

The topological classification results proved in this paper rely on the equivariant rigidity theorems for quasitoric and quoric manifolds established in \cite{mp} and \cite{gp}, respectively. In this section we recall the relevant constructions and explain how rigidity of the quotient spaces reduces the classification of moment--angle manifolds to the comparison of the associated principal bundles.

We now recall the notion of topological rigidity, which provides the conceptual framework for our approach (see \cite[Def.~7.1]{luck25}).

\begin{definition}[Topological Rigidity]
\label{def:rigidity}
Let $N$ be a closed topological manifold. We say that $N$ is \emph{topologically rigid} if every homotopy equivalence 
\[
f : M \longrightarrow N
\]
with $M$ a closed topological manifold as source is homotopic to a homeomorphism. 
\end{definition}

\begin{example}
The sphere $S^n$ is the prototypical example of a topologically rigid manifold.
The classical Poincaré Conjecture asserts that the $n$--sphere $S^n$ is topologically rigid.
\end{example}

\begin{remark}
Topological rigidity provides a broad generalization of the Poincar\'e Conjecture to manifolds with nontrivial fundamental groups. Classical rigidity results include:
\begin{itemize}
    \item[(1)] The Borel Conjecture for aspherical manifolds, asserting that manifolds with contractible universal cover are topologically rigid.
    \item[(2)] The Farrell--Jones Rigidity Theorem for negatively curved manifolds.
\end{itemize}
\end{remark}

Our goal is to establish a similar rigidity property for moment--angle manifolds, which in general are not aspherical. More precisely, we need the corresponding equivariant version for locally standard quotient manifolds.

\begin{definition}[Equivariant topological rigidity]
\label{def:eq_rigidity}
Let \(G\) be a compact Lie group, and let \(N\) be a closed \(G\)-manifold. We say that \(N\) is \emph{\(G\)-topologically rigid} if every \(G\)-equivariant homotopy equivalence
\[
    f:M\longrightarrow N
\]
from a closed \(G\)-manifold \(M\) is \(G\)-homotopic to a \(G\)-homeomorphism.
\end{definition}

The form of rigidity used in this paper is slightly weaker than the usual formulation of equivariant topological rigidity: we do not require the given equivariant homotopy equivalence itself to be equivariantly homotopic to a homeomorphism, but only that the existence of such an equivariant homotopy equivalence forces the two \(G\)-manifolds to be \(G\)-equivariantly homeomorphic. Thus this is a homeomorphism-type rigidity statement rather than a homotopy-to-homeomorphism statement.

\begin{definition}[Equivariant rigidity up to homeomorphism]
\label{def:eq_rigidity_weak}
Let \(G\) be a compact Lie group, and let \(N\) be a closed \(G\)-manifold. We say that \(N\) is \emph{\(G\)-rigid up to homeomorphism} if every \(G\)-equivariant homotopy equivalence
\[
    f:M\longrightarrow N
\]
from a closed \(G\)-manifold \(M\) implies the existence of a \(G\)-equivariant homeomorphism
\[
    h:M\longrightarrow N.
\]
\end{definition}

As a preliminary step, we recall that the quasitoric and quoric manifolds are rigid. Since the arguments in the two cases are closely related, we discuss them in parallel. The main difference arises from the noncommutativity of the quaternionic torus $(S^3)^n$, which makes the representation theory more involved than in the classical toric case.

\medskip

Throughout this section $G$ will denote either the torus $T^n=(S^1)^n$ or the quaternionic torus $Q^n=(S^3)^n$. The following Definition as well as details on locally linear actions can be found in \cite[p.171]{bredon}.

\begin{definition}
Let $N$ be a manifold with an action of a compact Lie group $G$. The action is called \emph{locally linear} if for every point $y\in N$ with isotropy subgroup $G_y$ there exists a $G_y$--invariant neighbourhood of $y$, a linear representation $V$ of $G_y$, and an open $G$--equivariant embedding
\[
G\times_{G_y} V \longrightarrow N .
\]
The image of such an embedding is called a \emph{linear tube}.
\end{definition}

Locally standard torus actions and locally regular quaternionic actions are locally linear (see \cite{mp,gp}). This condition allows one to analyze the orbit structure of the action using linear representations of the isotropy subgroups.

\medskip
Let $\mathcal Z_P$ be a moment--angle manifold equipped with its natural $G$--action, and let $N$ be a locally linear $G$--manifold. Suppose that there exists a $G$--equivariant homotopy equivalence
\[
f:N\longrightarrow \mathcal Z_P .
\]

Now we show that the orbit--type structure of $\mathcal Z_P$ transfers to $N$. In particular, the isotropy subgroups and the local linear models of the action on $N$ coincide with those of the model space.

\begin{prop}\label{prop:isotropy_transfer}
Let \(\mathcal Z_P\) be a moment--angle manifold with the canonical action of \(G=T^m\) respectively \(G=Q^m\), and let \(N\) be a locally linear \(G\)-manifold. Suppose that
\[
    f:N\longrightarrow \mathcal Z_P
\]
is a \(G\)-equivariant homotopy equivalence. Then the following hold:
\begin{enumerate}
\item[(1)] the \(G\)-action on \(N\) is effective;
\item[(2)] if \(H\) is an isotropy subgroup occurring in \(N\), then \(H\) is contained in an isotropy subgroup occurring in the canonical \(G\)-action on \(\mathcal Z_P\).
\end{enumerate}
\end{prop}

\begin{proof}
\begin{enumerate}
\item[(1)]
Suppose that some nontrivial element \(g\in G\) fixes \(N\) pointwise. Let \(H=\overline{\langle g\rangle}.\) We take the closure because fixed-point-set results for compact Lie group actions are applied to closed subgroups. Since the action is continuous and \(g\) fixes \(N\) pointwise, every element of \(H\) fixes \(N\) pointwise. Hence \(N^H=N.\) Since \(f\) is a \(G\)-equivariant homotopy equivalence and \(H\) is closed, it induces a homotopy equivalence
\[
f^H:N^H\longrightarrow \mathcal Z_P^H
\]
by \cite[Cor.~II.5.5]{bredon}. Thus \(\mathcal Z_P^H\) is homotopy equivalent to \(N\). But \(H\) is nontrivial and the canonical \(G\)-action on \(\mathcal Z_P\) is effective, so \(\mathcal Z_P^H\) is a proper fixed-point subspace of \(\mathcal Z_P\). This contradicts \(N^H=N\). Hence the action on \(N\) is effective.

\item[(2)]
Let \(x\in N\), and let \(H=G_x\). Since the action is locally linear, \(H\) is a closed subgroup of \(G\). Moreover, \(x\in N^H\), so \(N^H\neq\varnothing\). Applying \cite[Cor.~II.5.5]{bredon} again, we obtain a homotopy equivalence
\[
f^H:N^H\longrightarrow \mathcal Z_P^H.
\]
Hence \(\mathcal Z_P^H\neq\varnothing\). Therefore there exists a point \(z\in \mathcal Z_P\) fixed by \(H\), which means that
\[
    H\subseteq G_z.
\]
The isotropy subgroups \(G_z\) of the canonical \(G\)-action on \(\mathcal Z_P\) are the coordinate subgroups associated to faces of \(P\). Hence \(H\) is contained in a coordinate isotropy subgroup of the model action.
\end{enumerate}
\end{proof}

\begin{cor}\label{cor:isotropy_equal}
The isotropy subgroups of the $G$--actions on $N$ and $\mathcal Z_P$ are the same.
\end{cor}

\medskip

In both cases there exists a subgroup $K \subset G$ acting freely on $\mathcal Z_P$ such that the quotient
\[
\mathcal Z_P/K
\]
is a quasitoric (respectively quoric) manifold over $P$.

Let $N$ be a locally linear $G$--manifold equipped with a $G$--equivariant homotopy equivalence
\[
f:N\longrightarrow \mathcal Z_P .
\]

The following Lemma is the consequence of the fixed-point argument that will be used in the sequel.

\begin{lemma}\label{lem:Kfree}
The subgroup $K$ acts freely on $N$.
\end{lemma}
\begin{proof}
Suppose that some nontrivial element $k\in K$ fixes a point $x\in N$. Let $H=\overline{\langle k\rangle}$. Then \(H\) is a closed subgroup of \(G\). Since the action is continuous, $N^H=N^{\langle k\rangle}$. Then $x\in N^H$, so the fixed point set $N^H$ is nonempty.

Since $f$ is a $G$--equivariant homotopy equivalence, it induces a homotopy equivalence
\[
f^H:N^H\longrightarrow \mathcal Z_P^H .
\]

However, by construction the subgroup $K$ acts freely on $\mathcal Z_P$, so
\[
\mathcal Z_P^H=\varnothing
\]
whenever $H$ is a nontrivial subgroup of $K$. This contradicts the fact that $N^H$ is nonempty. Therefore no nontrivial element of $K$ can fix a point of $N$, and the $K$--action on $N$ is free.
\end{proof}

Since the action of $K$ on both $N$ and $\mathcal Z_P$ is free, the map $f$ descends to a map between the quotient spaces.

\begin{prop}\label{prop:pushdown}
The equivariant homotopy equivalence
\[
f:N\longrightarrow \mathcal Z_P
\]
induces a homotopy equivalence
\[
\bar f:N/K\longrightarrow \mathcal Z_P/K .
\]
\end{prop}

\begin{proof}
Since $f$ is $G$--equivariant, it is in particular $K$--equivariant. Therefore it descends to a continuous map between the quotient spaces. Because $f$ is a homotopy equivalence and $K$ acts freely on both spaces, the induced map $\bar f$ is also a homotopy equivalence.
\end{proof}

The quotient $\mathcal Z_P/K$ is a quasitoric manifold in the complex case and a quoric manifold in the quaternionic case. The induced map $\bar f$ is equivariant with respect to the residual action of $G/K$.

We now verify the hypotheses needed in order to apply the equivariant rigidity results for quasitoric \cite{mp} and quoric manifolds \cite{gp}. Namely, we show that the quotients by the free kernel action are closed locally linear manifolds of the correct dimensions, \(2n\) in the complex case and \(4n\) in the quaternionic case, carrying the residual \(T^n\)- and \(Q^n\)-actions, respectively.

\begin{lemma}\label{lem:dim_match}
Let \(X\) and \(Y\) be closed topological manifolds, and suppose that \(f:X\to Y\) is a homotopy equivalence. Then
\[
    \dim X=\dim Y.
\]
\end{lemma}
\begin{proof}
We use cohomology with \(\mathbb Z/2\)-coefficients.  If \(X\) is a closed \(d\)-dimensional topological manifold, then
\[
    H^d(X;\mathbb Z/2)\neq 0
\]
and \(H^i(X;\mathbb Z/2)=0\) for all \(i>d\).  Since \(f\) is a homotopy equivalence, it induces an isomorphism on cohomology in all degrees. If \(\dim X\neq \dim Y\), say \(\dim X>\dim Y\), then
\[
    H^{\dim X}(Y;\mathbb Z/2)=0
\]
whereas
\[
    H^{\dim X}(X;\mathbb Z/2)\neq 0,
\]
contradicting the cohomology isomorphism. Hence \(\dim X=\dim Y\).
\end{proof}

In particular, if \(N\) is a closed manifold equipped with a homotopy equivalence \(f:N\longrightarrow \mathcal Z_P\) or \(f:N\longrightarrow \mathcal Z_P^{\mathbb H}\), then \(\dim N=\dim \mathcal Z_P=m+n\) in the complex case, and \(\dim N=\dim \mathcal Z_P^{\mathbb H}=3m+n\) in the quaternionic case, respectively.

\begin{prop}\label{prop:quotient_manifold}
Let \(G=T^m\) in the complex case and \(G=Q^m\) in the quaternionic case.  Let \(N\) be a closed locally linear \(G\)-manifold equipped with a \(G\)-equivariant homotopy equivalence
\[
    f:N\longrightarrow \mathcal Z_P
\]
in the complex case, or
\[
    f:N\longrightarrow \mathcal Z_P^{\mathbb H}
\]
in the quaternionic case. Let \(K\subset G\) be the kernel subgroup associated to the characteristic data. By Lemma~\ref{lem:Kfree}, \(K\) acts freely on \(N\). Then \(N/K\) is a closed topological manifold and the residual \(G/K\)-action on \(N/K\) is locally linear.

Moreover, \(\dim(N/K)=2n\) in the complex case, and
\(\dim(N/K)=4n\) in the quaternionic case.
\end{prop}

\begin{proof}
Since \(K\) is a compact Lie group acting freely and locally linearly on the closed manifold \(N\), the quotient map
\[
    N\longrightarrow N/K
\]
is a locally trivial principal \(K\)-bundle. In particular, \(N/K\) is a closed topological manifold and
\[
    \dim(N/K)=\dim N-\dim K.
\]

By Lemma~\ref{lem:dim_match}, in the complex case
\[
    \dim N=\dim\mathcal Z_P=m+n.
\]
Since \(K\cong T^{m-n}\), we obtain
\[
    \dim(N/K)
    =
    (m+n)-(m-n)
    =
    2n.
\]

In the quaternionic case,
\[
    \dim N=\dim\mathcal Z_P^{\mathbb H}=3m+n.
\]
Since \(K\cong Q^{m-n}=(S^3)^{m-n}\), we have
\[
\dim K=3(m-n),
\]
and therefore
\[
\dim(N/K)=(3m+n)-3(m-n)=4n.
\]

It remains to check local linearity of the residual action. Let \([x]\in N/K\), and let \(G_x\) be the isotropy subgroup of \(x\) in \(G\). Since \(K\) acts freely on \(N\), we have
\[
G_x\cap K=\{1\}.
\]
By local linearity of the \(G\)-action on \(N\), there is a \(G\)-invariant neighbourhood of the orbit \(Gx\) of the form
\[
    G\times_{G_x} V,
\]
where \(V\) is a linear \(G_x\)-representation.  Quotienting by the free \(K\)-action gives
\[
(G\times_{G_x}V)/K\cong (G/K)\times_{G_x}V,
\]
where \(G_x\) is identified with its image in \(G/K\). This is a linear tube for the residual \(G/K\)-action near \([x]\). Hence the residual action of \(G/K\) on \(N/K\) is locally linear.
\end{proof}

We now record the consequence needed in order to apply the rigidity theorems. Proposition~\ref{prop:quotient_manifold} verifies the manifold-theoretic hypotheses on the source quotient, while Proposition~\ref{prop:pushdown} gives the descended equivariant homotopy equivalence.

\begin{cor}\label{cor:quotient_satisfies_rigidity_hypotheses}
The quotient \(N/K\) satisfies the manifold-theoretic hypotheses of the corresponding equivariant rigidity theorem. More precisely, \(N/K\) is a closed locally linear \(T^n\)-manifold of dimension \(2n\) in the complex case, and a closed locally linear \(Q^n\)-manifold of dimension \(4n\) in the quaternionic case. Moreover, the descended map \(\bar f\) is an equivariant homotopy equivalence onto the corresponding quotient, \(\mathcal Z_P/K\) or \(\mathcal Z_P^{\mathbb H}/K\), respectively.
\end{cor}

We can therefore apply the equivariant rigidity results for quasitoric and quoric manifolds. In the complex case and quasitoric manifolds we have:

\medskip

\begin{thm}[Topological rigidity of quasitoric manifolds \cite{mp}]\label{thm:top_rigidity_quasitoric}
Let $M^{2n}$ be a quasitoric manifold and let $N$ be a locally linear $T^n$--manifold. If
\[
f:N\longrightarrow M^{2n}
\]
is a $T^n$--equivariant homotopy equivalence, then $f$ is equivariantly homotopic to a $T^n$--homeomorphism.
\end{thm}

\begin{cor}\label{cor:quotient_quasitoric}
The assumptions of Theorem \ref{thm:top_rigidity_quasitoric} imply that the manifold $N$ is a quasitoric manifold.
\end{cor}

\medskip

In the quaternionic case and quoric manifolds we have the following results: 

\medskip

\begin{thm}[Topological rigidity of quoric manifolds \cite{gp}]\label{thm:top_rigidity_quoric}
Let $M^{4n}$ be a quoric manifold and let $N$ be a locally linear $Q^n$--manifold. If
\[
f:N\longrightarrow M^{4n}
\]
is a $Q^n$--equivariant homotopy equivalence, then $f$ is equivariantly homotopic to a $Q^n$--homeomorphism.
\end{thm}

\begin{cor}\label{cor:quotient_quoric}
The assumptions of Theorem \ref{thm:top_rigidity_quoric} imply that the manifold $N$ is a quoric manifold.
\end{cor}

\medskip

Summarizing all the above, we obtain the following immediate result.

\begin{cor}\label{cor:quotient_homeo}
The quotient manifolds $N/K$ and $\mathcal Z_P/K$ are equivariantly homeomorphic. 
\end{cor}

\medskip

The classification problem for $N$ therefore reduces to understanding the possible principal $K$--bundles over the quasitoric or quoric manifold $\mathcal Z_P/K$.

This reduction will be used in the following section to complete the classification of moment--angle manifolds.

\section{Topological classification of moment--angle manifolds}
\label{sec:classification}

Throughout this section \(G\) denotes \(T^m\) in the complex case and \(Q^m\) in the quaternionic case.  We write
\[
    \mathcal Z=
    \begin{cases}
    \mathcal Z_P, & \text{in the complex case},\\
    \mathcal Z_P^{\mathbb H}, & \text{in the quaternionic case}.
    \end{cases}
\]
Let \(K\subset G\) be the kernel subgroup associated to the chosen characteristic data, as in Section~\ref{sec:reduction}. We write
\[
    M=\mathcal Z/K
\]
for the corresponding quasitoric, respectively quoric, quotient, and
\[
    q:\mathcal Z\longrightarrow M
\]
for the quotient map.

The aim of this section is to prove that \(\mathcal Z\) is \(G\)-rigid up to homeomorphism in the sense of Definition~\ref{def:eq_rigidity_weak}. The proof uses the reduction results of Section~\ref{sec:reduction} together with a pullback argument for the associated principal \(K\)-bundles.  We emphasize that this is a homeomorphism-type rigidity statement: we do not claim that the original equivariant homotopy equivalence is itself equivariantly homotopic to a homeomorphism.

\subsection{The pullback construction}

We shall use the following standard fact about principal bundles over a fixed base.

\begin{lemma}\label{lem:bundle_map}
Let \(K\) be a compact Lie group, and let
\[
    p:E\to B,
    \qquad
    p':E'\to B
\]
be principal \(K\)-bundles over the same base.  If
\[
    u:E\longrightarrow E'
\]
is a \(K\)-equivariant continuous map covering \(\mathrm{id}_B\), then \(u\) is an isomorphism of principal \(K\)-bundles.
\end{lemma}

\begin{proof}
This is a standard fact about equivalence of principal bundles over a fixed base; see Steenrod \cite[Ch.~I, \S2]{Steenrod}. For completeness, we recall the argument.  For each \(b\in B\), the restriction
\[
    u_b:E_b\longrightarrow E'_b
\]
is a \(K\)-equivariant map between two free transitive \(K\)-spaces. Hence \(u_b\) is a bijection. In local trivializations, \(u\) is given by multiplication by a continuous \(K\)-valued function, and is therefore a local homeomorphism. Thus \(u\) is a homeomorphism and hence an isomorphism of principal \(K\)-bundles.
\end{proof}

\begin{prop}[Pullback identification]
\label{prop:pullback_construction}
Let \(N\) be a closed locally linear \(G\)-manifold, and let
\[
    f:N\longrightarrow \mathcal Z
\]
be a \(G\)-equivariant homotopy equivalence. Let
\[
    \bar N=N/K,
\]
let
\[
    \bar f:\bar N\longrightarrow M
\]
be the induced map of Proposition~\ref{prop:pushdown}, and let
\[
    \bar f^*\mathcal Z
    =
    \{(x,z)\in \bar N\times \mathcal Z \mid \bar f(x)=q(z)\}
\]
be the pullback of \(q:\mathcal Z\to M\) along \(\bar f\).  Then the map
\[
    F:N\longrightarrow \bar f^*\mathcal Z,
    \qquad
    F(z)=\bigl(q_N(z),f(z)\bigr),
\]
where \(q_N:N\to \bar N\) is the quotient map, is a \(G\)-equivariant homeomorphism.
\end{prop}

\begin{proof}
The maps involved fit into the following commutative diagram:

\[
\begin{tikzcd}[column sep=large,row sep=large]
N
    \arrow[r,"F"]
    \arrow[d,"q_N"']
    \arrow[rr,bend left=20,"f"]
&
\bar f^{\,*}\mathcal Z
    \arrow[d,"\mathrm{pr}_1"]
    \arrow[r,"\mathrm{pr}_2"]
&
\mathcal Z
    \arrow[d,"q"]
\\
\bar N
    \arrow[r,equal]
&
\bar N
    \arrow[r,"\bar f"']
&
M .
\end{tikzcd}
\]

Here \(F(z)=\bigl(q_N(z),f(z)\bigr)\). The map \(F\) is well defined because \(q(f(z))=\bar f(q_N(z))\). Moreover, \(\mathrm{pr}_1\circ F=q_N\), \(\mathrm{pr}_2\circ F=f\). The \(G\)-action on the pullback is given by
\[
    g\cdot(x,z)=(\bar g x,gz),
\]
where \(\bar g\) denotes the image of \(g\) in \(G/K\). Hence
\[
    F(gz)
    =
    \bigl(q_N(gz),f(gz)\bigr)
    =
    \bigl(\bar g\,q_N(z),g f(z)\bigr)
    =
    g\cdot F(z).
\]
Thus \(F\) is \(G\)-equivariant.

Both \(q_N:N\to \bar N\) and \(\bar f^*\mathcal Z\to \bar N\) are principal \(K\)-bundles. The second is the pullback of the principal \(K\)-bundle \(q:\mathcal Z\to M\). The map \(F\) covers \(\mathrm{id}_{\bar N}\), and it is \(K\)-equivariant because \(f\) is \(G\)-equivariant. Hence Lemma~\ref{lem:bundle_map} applies, and \(F\) is an isomorphism of principal \(K\)-bundles. In particular, \(F\) is a \(G\)-equivariant homeomorphism.
\end{proof}

\subsection{Equivariant homotopy invariance of pullbacks}

All principal bundles appearing below are locally trivial bundles over closed manifolds, hence over paracompact spaces. Therefore they are numerable. We may consequently apply the homotopy invariance theorem for numerable principal bundles (in the nonequivariant setting; see \cite{husemoller} and for the equivariant form; see \cite{lashofmay}.).

\begin{lemma}[Equivariant pullback lemma]
\label{lem:equiv_pullback}
Let
\(q:\mathcal Z\to M\) be the \(G\)-equivariant principal \(K\)-bundle above. Suppose that \(u_0,u_1:\bar N\to M\) are \(G/K\)-equivariantly homotopic. Then the pullback principal \(K\)-bundles \(u_0^*\mathcal Z\) and \(u_1^*\mathcal Z\) are \(G\)-equivariantly isomorphic.
\end{lemma}
\begin{proof}
Let
\[
    H:\bar N\times I\longrightarrow M
\]
be a \(G/K\)-equivariant homotopy from \(u_0\) to \(u_1\), where \(G/K\) acts trivially on \(I\). Pulling back \(q:\mathcal Z\to M\) along \(H\), we obtain a \(G\)-equivariant principal \(K\)-bundle
\[
    H^*\mathcal Z\longrightarrow \bar N\times I.
\]
Its restrictions to \(\bar N\times\{0\}\) and \(\bar N\times\{1\}\) are precisely \(u_0^*\mathcal Z\) and \(u_1^*\mathcal Z\).

By homotopy invariance of numerable principal bundles, these two restrictions are isomorphic as principal \(K\)-bundles; see Husemoller \cite[Ch.~4, Theorem~9.9]{husemoller}. Since the homotopy \(H\) is \(G/K\)-equivariant and the original bundle \(q:\mathcal Z\to M\) is \(G\)-equivariant, the pullback bundle \(H^*\mathcal Z\) carries a compatible \(G\)-action. The equivariant form of the same homotopy-invariance result for numerable principal bundles gives the isomorphism \(G\)-equivariantly; cf. \cite[Corollary~7]{lashofmay}. Therefore
\[
    u_0^*\mathcal Z\cong_G u_1^*\mathcal Z.
\]
\end{proof}

\begin{lemma}
\label{lem:pullback_homeo}
If \(h:\bar N\to M\) is a \(G/K\)-equivariant homeomorphism, then
\[
    \mathrm{pr}_2:h^*\mathcal Z\longrightarrow \mathcal Z,
    \qquad
    (x,z)\mapsto z,
\]
is a \(G\)-equivariant homeomorphism.
\end{lemma}

\begin{proof}
The inverse is given by $$z\longmapsto \bigl(h^{-1}(q(z)),z\bigr).$$ This map is continuous and inverse to \(\mathrm{pr}_2\). Equivariance follows from
\[
    \mathrm{pr}_2(g\cdot(x,z))
    =
    \mathrm{pr}_2(\bar g x,gz)
    =
    gz
    =
    g\cdot \mathrm{pr}_2(x,z).
\]
\end{proof}

\subsection{The main rigidity theorem}

Let \(\mathcal Z\) be a complex or quaternionic moment--angle manifold associated to a simple polytope \(P^n\), with canonical \(G\)-action, where \(G=T^m\) in the complex case and \(G=Q^m\) in the quaternionic case. Suppose that \(P\) is equipped with characteristic data giving a kernel subgroup \(K\subset G \) such that \(M=\mathcal Z/K \) is a quasitoric manifold in the complex case, respectively a quoric manifold in the quaternionic case. In the quaternionic case, assume that the characteristic data are global.

We now combine the reduction to the quasitoric and quoric quotients with the pullback construction above to obtain the main rigidity theorem.

\begin{thm}
\label{thm:main_rigidity}
Let \(N\) be a closed locally linear \(G\)-manifold, and let \(f:N\longrightarrow \mathcal Z \) be a \(G\)-equivariant homotopy equivalence. Then \(N\) is \(G\)-equivariantly homeomorphic to \(\mathcal Z\).
\end{thm}

\begin{proof}
By Lemma~\ref{lem:Kfree}, the subgroup \(K\) acts freely on \(N\). By Propositions~\ref{prop:pushdown} and~\ref{prop:quotient_manifold}, the quotient \(\bar N=N/K\) is a closed locally linear \(G/K\)-manifold of the correct dimension, and the induced map \(\bar f:\bar N\longrightarrow M \) is a \(G/K\)-equivariant homotopy equivalence. By Corollary~\ref{cor:quotient_satisfies_rigidity_hypotheses}, the pair \((\bar N,\bar f)\) satisfies the hypotheses of the corresponding equivariant rigidity theorem.

Applying Theorem~\ref{thm:top_rigidity_quasitoric} in the complex case and Theorem~\ref{thm:top_rigidity_quoric} in the quaternionic case, we obtain a \(G/K\)-equivariant homeomorphism \(h:\bar N\longrightarrow M \) which is \(G/K\)-equivariantly homotopic to \(\bar f\).

By Proposition~\ref{prop:pullback_construction}, there is a \(G\)-equivariant homeomorphism \(N\cong_G \bar f^*\mathcal Z\). Since \(\bar f\) and \(h\) are \(G/K\)-equivariantly homotopic, the equivariant pullback Lemma \ref{lem:equiv_pullback} gives \(\bar f^*\mathcal Z\cong_G h^*\mathcal Z\). Since \(h\) is a \(G/K\)-equivariant homeomorphism, Lemma~\ref{lem:pullback_homeo} gives \(h^*\mathcal Z\cong_G \mathcal Z\). Combining these three \(G\)-equivariant homeomorphisms, we obtain
\[
    N
    \cong_G
    \bar f^*\mathcal Z
    \cong_G
    h^*\mathcal Z
    \cong_G
    \mathcal Z.
\]
Therefore \(N\) is \(G\)-equivariantly homeomorphic to \(\mathcal Z\).
\end{proof}

The following two corollaries record the complex and quaternionic specializations of Theorem~\ref{thm:main_rigidity}.

\begin{cor}[Rigidity of complex moment--angle manifolds]
\label{cor:rigidity_complex}
Let \(P^n\) be a simple polytope admitting characteristic data whose kernel subgroup \(K\subset T^m\) acts freely on \(\mathcal Z_P\), so that
\[
    \mathcal Z_P/K
\]
is a quasitoric manifold. Then every closed locally linear \(T^m\)-manifold that is \(T^m\)-equivariantly homotopy equivalent to \(\mathcal Z_P\) is \(T^m\)-equivariantly homeomorphic to \(\mathcal Z_P\).
\end{cor}

\begin{cor}[Rigidity of quaternionic moment--angle manifolds]
\label{cor:rigidity_quaternionic}
Let \(P^n\) be a simple polytope admitting a global characteristic pair. Then every closed locally linear \(Q^m\)-manifold that is \(Q^m\)-equivariantly homotopy equivalent to \(\mathcal Z_P^{\mathbb H}\) is \(Q^m\)-equivariantly homeomorphic to \(\mathcal Z_P^{\mathbb H}\), for every \(n\).
\end{cor}

\begin{remark}[The strong form]
Theorem~\ref{thm:main_rigidity} proves \(G\)-rigidity up to homeomorphism in the sense of Definition~\ref{def:eq_rigidity_weak}. The \(G\)-equivariant homeomorphism \(N\cong_G \mathcal Z\) constructed above is obtained by combining the pullback identification, equivariant homotopy invariance of pullbacks, and the pullback along the quotient homeomorphism. The argument does not show that this homeomorphism is \(G\)-equivariantly homotopic to the original map \(f\).  Thus the stronger homotopy-to-homeomorphism form of equivariant topological rigidity remains outside the scope of the present paper.
\end{remark}

\paragraph*{Further discussion.}
Theorem~\ref{thm:main_rigidity} shows that complex moment--angle manifolds, and quaternionic moment--angle manifolds, are \(G\)-rigid up to homeomorphism: their equivariant homotopy type determines their equivariant homeomorphism type. This result is in the spirit of Borel-type rigidity, but it is not an ordinary Borel rigidity theorem in the sense of Kreck and L\"uck \cite{kl09}. Indeed, ordinary Borel rigidity concerns arbitrary nonequivariant homotopy equivalences and requires such maps to be homotopic to homeomorphisms. By contrast, our result is equivariant and proves the existence of an equivariant homeomorphism without controlling its homotopy relation to the original map \(f\).

It remains an interesting question whether the stronger homotopy-to-homeomorphism form holds for moment--angle manifolds in this equivariant setting. One possible approach would be to track an explicit equivariant homotopy through the covering-homotopy construction underlying Lemma~\ref{lem:equiv_pullback}.

\end{document}